\documentclass[12pt]{amsart}
\usepackage{amssymb}

\let\cal\mathcal
\let\frak\mathfrak
\let\Bbb\mathbb

\def\>{\relax\ifmmode\mskip.666667\thinmuskip\relax\else\kern.111111em\fi}
\def\<{\relax\ifmmode\mskip-.333333\thinmuskip\relax\else\kern-.0555556em\fi}
\def\vsk#1>{\vskip#1\baselineskip}
\def\vv#1>{\vadjust{\vsk#1>}\ignorespaces}
\def\vvn#1>{\vadjust{\nobreak\vsk#1>\nobreak}\ignorespaces}
\def\vvgood{\vadjust{\penalty-500}} \let\alb\allowbreak
\let\dsty\displaystyle \let\tsty\textstyle
 \let\sssty\scriptscriptstyle
\def\fratop{\genfrac{}{}{0pt}1}

\def\tsum{\mathop{\tsty\sum}\limits}

\def\sskip{\par\vsk.2>}
\let\Medskip\medskip
\def\medskip{\par\Medskip}
\let\Bigskip\bigskip
\def\bigskip{\par\Bigskip}

\let\Maketitle\maketitle
\def\maketitle{\Maketitle\thispagestyle{empty}\let\maketitle\empty}

\newtheorem{thm}{Theorem}[section]
\newtheorem{cor}[thm]{Corollary}
\newtheorem{lem}[thm]{Lemma}

\numberwithin{equation}{section}

\theoremstyle{definition}
\newtheorem*{rem}{Remark}
\newtheorem*{example}{Example}

\def\beq{\begin{equation}}
\def\eeq{\end{equation}}
\def\be{\begin{equation*}}
\def\ee{\end{equation*}}

\def\bean{\begin{eqnarray}}
\def\eean{\end{eqnarray}}
\def\bea{\begin{eqnarray*}}
\def\eea{\end{eqnarray*}}
\def\Ref#1{{\rm(\ref{#1})}}

\let\al\alpha

\let\gm\gamma  
\let\dl\delta \let\Dl\Delta 
 \let\eps\varepsilon \let\epsilon\eps

\let\la\lambda 

\let\si\sigma 
 
\let\pho\phi \let\phi\varphi

 \let\Om\Omega

\let\Hat\widehat
\let\der\partial

\let\ge\geqslant

\let\le\leqslant

\let\ox\otimes

\def\C{{\Bbb C}}

\def\Z{{\Bbb Z}}

\def\gl{\frak{gl}}

\def\Bc{{\cal B}}
\def\Cc{{\cal C}}

\def\Ec{{\cal E}}
\def\Fc{{\cal F}}

\def\Mc{{\cal M}}

\def\Fh{\Hat\Fc}
\def\Qh{\Hat Q}

\def\lsym#1{#1\alb\ldots\relax#1\alb} \def\lc{\lsym,}

\let\on\operatorname

\def\diag{\on{diag}}
\def\End{\on{End}}

\def\rdet{\on{rdet}}

\def\tr{\on{tr}}
\def\Wr{\on{Wr}}

\def\sdag{{\sssty\dag}}

\def\gln{\gl_N}
\def\glnt{\gln[t]}
\def\Ugln{U(\gln)}

\def\Uglnt{U(\glnt)}

\let\bs\boldsymbol
\let\mc\mathcal
\let\nc\newcommand

\nc{\pone}{\C{\mathbb P}^1}

\def\B{{\mc B}}
\def\D{{\mc D}}

\nc{\g}{{\mathfrak g}}
\nc{\h}{{\mathfrak h}}
\nc{\m}{{\mathfrak m}}
\nc{\n}{{\mathfrak n}}

\def\reg{{\it reg}}
\def\sing{{\it\>sing}}
\nc{\ep}{\epsilon}

\def\KZ/{{\sl KZ\/}}

\let\kk K
\let\yh h

\begin{document}

\hrule width0pt
\vsk->

\title[Gaudin Hamiltonians generate the Bethe algebra]
{Gaudin Hamiltonians generate the Bethe algebra of\\
a tensor power of vector representation of {\large $\gln$}}

\author[E.\,Mukhin, V.\,Tarasov, and \>A.\,Varchenko]
{E.\,Mukhin$\>^*$, V.\,Tarasov$\>^\star$, and \>A.\,Varchenko$\>^\diamond$}

\maketitle

\begin{center}
\vsk-.2>
{\it $\kern-.4em^{*,\star}\<$Department of Mathematical Sciences,
Indiana University\,--\>Purdue University Indianapolis\kern-.4em\\
402 North Blackford St, Indianapolis, IN 46202-3216, USA\/}

\medskip
{\it $^\star\<$St.\,Petersburg Branch of Steklov Mathematical Institute\\
Fontanka 27, St.\,Petersburg, 191023, Russia\/}

\medskip
{\it $^\diamond\<$Department of Mathematics, University of North Carolina
at Chapel Hill\\ Chapel Hill, NC 27599-3250, USA\/}
\end{center}

{\let\thefootnote\relax
\footnotetext{\vsk-.8>\noindent
$^*$\,Supported in part by NSF grant DMS-0601005\\
$^\star$\,Supported in part by RFFI grant 08-01-00638\\
$^\diamond$\,Supported in part by NSF grant DMS-0555327}}

\medskip
\begin{abstract}
We show that the Gaudin Hamiltonians $H_1\lc H_n$ generate the Bethe algebra
of the $n$-fold tensor power of the vector representation of $\gln$.
Surprisingly the formula for the generators of the Bethe algebra in terms
of the Gaudin Hamiltonians does not depend on $N$. Moreover, this formula
coincides with Wilson's formula for the stationary Baker-Akhiezer function
on the adelic Grassmannian.
\end{abstract}

\section{Introduction}
The Gaudin model describes a completely integrable quantum spin chain
\cite{G1}, \cite{G2}. We consider the Gaudin model associated with
the Lie algebra $\gln$. Denote by
$L_{\bs\la}$ the irreducible finite-dimensional $\gln$-module with
highest weight $\bs\la$. Consider a tensor product
$\otimes_{a=1}^nL_{\bs\la^{(a)}}$ of such modules and two sequences
of complex numbers: \,$\kk_1\lc\kk_N$ and $z_1\lc z_n$.
Assume that the numbers $z_1\lc z_n$ are distinct.
The Hamiltonians of the Gaudin model are mutually commuting operators
$H_1\lc H_n$, acting on the space $\otimes_{a=1}^nL_{\bs\la^{(a)}}$,
\vvn.2>
\beq
\label{Ha}
H_a\ =\ \sum_{i=1}^N\kk_i\>e_{ii}^{(a)}\,+\>
\sum_{i,j=1}^N\,\sum_{b\neq a} \,\frac{e_{ij}^{(a)}e_{ji}^{(b)}}{z_a-z_b}\;,
\eeq
where \>$e_{ij}$ are the standard generators of $\gln$ and \>$e_{ij}^{(a)}$
is the image of $1^{\otimes(a-1)}\otimes e_{ij}\otimes1^{\otimes(n-a)}$.

One of the main problems in the Gaudin model is to find eigenvalues and
joint eigenvectors of the operators $H_1\lc H_n$, see \cite{B}, \cite{RV},
\cite{MTV1}. The Gaudin Hamiltonians appear also as the right-hand sides
of the Knizhnik-Zamolodchikov equations, see~\cite{SV}, \cite{RV}, \cite{FFR},
\cite{FMTV}.

It was realized long time ago that there are additional interesting
operators commuting with the operators $H_1\lc H_n$, see for example~\cite{KS},
\cite{FFR}. Those operators are called the higher Gaudin Hamiltonians.
To distinguish the operators $H_1\lc H_n$, we will call them the classical
Gaudin Hamiltonians.

The algebra generated by all of the classical and higher Gaudin Hamiltonians is
called the Bethe algebra. A useful formula for generators of the Bethe algebra
was suggested in \cite{T}, see also \cite{MTV1}, \cite{CT}.

In general, the Bethe algebra is larger than its subalgebra generated by
the classical Gaudin Hamiltonians. Nevertheless, we show in this paper that
if all factors of the tensor product $\otimes_{a=1}^nL_{\bs\la^{(a)}}$ are
the standard vector representations of $\gln$, then the classical Gaudin
Hamiltonians generate the entire Bethe algebra. It is a surprising fact since
every tensor product of polynomial $\gln$-modules is a submodule of a tensor
power of the vector representation,
and one may expect that the Bethe algebra of a tensor power of the vector
representation is as general as the Bethe algebra of a tensor product of
arbitrary representations. Another surprising fact is that our formula for
the elements of the Bethe algebra in terms of the classical Gaudin Hamiltonians
does not depend on $N$, see Theorem~\ref{generate}. The third surprise is
that our formula is nothing else but Wilson's formula for the stationary
Baker-Akhiezer function on the adelic Grassmannian \cite{Wi}.

Our theorem can be used to study the higher Gaudin Hamiltonians as functions
of the classical Hamiltonians (or as limits of functions of the classical
Gaudin Hamiltonians). It is known much more about the classical Hamiltonians
than about the higher Hamiltonians.

Our proof of Theorem~\ref{generate} is not elementary. We use the fact that
the Bethe algebra is preserved under the $(\gln,\frak{gl}_n)$-duality and the
completeness of the Bethe ansatz for a tensor product of vector representations
and generic $\kk_1\lc\kk_N$, $z_1\lc z_n$.

\section{Bethe algebra}
\label{alg sec}
\subsection{Lie algebras $\gln$ and $\glnt$}
Let $e_{ij}$, $i,j=1\lc N$, be the standard generators of the Lie algebra
$\gln$ satisfying the relations
$[e_{ij},e_{sk}]=\dl_{js}e_{ik}-\dl_{ik}e_{sj}$. Let $\h\subset\gln$ be
the Cartan subalgebra generated by $e_{ii}, \,i=1\lc N$.

We denote by $V=\oplus_{i=1}^N\C v_i$ the standard $N$-dimensional vector
representation of $\gln$: \,$e_{ij}v_j=v_i$ \,and \,$e_{ij}v_k=0$ \>for
\>$j\ne k$.

Let $M$ be a $\gln$-module.
A vector $v\in M$ is called {\it singular\/} if $e_{ij}v=0$ for
$1\le i<j\le N$. We denote by $M^\sing$ the subspace of all singular vectors
in $M$.

Let $\glnt=\gln\otimes\C[t]$ be the complex Lie algebra of $\gln$-valued
polynomials with the pointwise commutator. For $g\in\gln$, we set
$g(u)=\sum_{s=0}^\infty (g\otimes t^s)u^{-s-1}$.

We identify $\gln$ with the subalgebra $\gln\otimes1$ of constant polynomials
in $\glnt$. Hence, any $\glnt$-module has a canonical structure of
a $\gln$-module.

For each $a\in\C$, there exists an automorphism $\rho_a$ of $\glnt$,
\;$\rho_a:g(u)\mapsto g(u-a)$. Given a $\glnt$-module $M$, we denote by $M(a)$
the pull-back of $M$ through the automorphism $\rho_a$. As $\gln$-modules,
$M$ and $M(a)$ are isomorphic by the identity map.

We have the evaluation homomorphism,
${\glnt\to\gln}$, \;${g(u) \mapsto g\>u^{-1}}$.
Its restriction to the subalgebra $\gln\subset\glnt$ is the identity map.
For any $\gln$-module $M$, we denote by the same letter the $\glnt$-module,
obtained by pulling $M$ back through the evaluation homomorphism.

\subsection{Bethe algebra}
\label{secbethe}
Given an ${N\times N}$-matrix $A$ with possibly noncommuting entries $a_{ij}$,
we define its {\it row determinant\/} to be
\vvn.3>
\be
\rdet A\,=
\sum_{\;\si\in S_N\!} (-1)^\si\,a_{1\si(1)}a_{2\si(2)}\ldots a_{N\si(N)}\,.
\vv-.1>
\ee

Let $\kk_1\lc\kk_N$ be a sequence of complex numbers.
Let $\der_u$ be the operator of differentiation in the variable $u$.
Define the {\it universal differential operator\/} $\D^\kk$
\vvn.4>
by the formula
\vvgood
\beq
\label{DK}
\D^\kk=\,\rdet\left( \begin{matrix}
\der_u-\kk_1-e_{11}(u) & -\>e_{21}(u)& \dots & -\>e_{N1}(u)\\[3pt]
-\>e_{12}(u) &\der_u-\kk_2-e_{22}(u)& \dots & -\>e_{N2}(u)\\[1pt]
\dots & \dots &\dots &\dots \\[1pt]
-\>e_{1N}(u) & -\>e_{2N}(u)& \dots & \der_u-\kk_N-e_{NN}(u)
\end{matrix}\right).
\vv.2>
\eeq
It is a differential operator in $u$, whose coefficients are
formal power series in $u^{-1}$ with coefficients in $\Uglnt$,
\vvn-.1>
\be
\D^\kk=\,\der_u^N+\sum_{i=1}^N\,B_i^\kk(u)\,\der_u^{N-i}\>,
\qquad
B_i^\kk(u)\,=\,\sum_{j=0}^\infty B_{ij}^\kk\>u^{-j}\,,
\ee
and $B_{ij}^\kk\in\Uglnt$, \,$i=1\lc N$, \,$j\in\Z_{\ge 0}\>$. We have
\beq
\label{Bi0}
\der_u^N\>+\>\sum_{i=1}^N B_{i0}^\kk\,\der_u^{N-i}\>=
\,\prod_{i=1}^N\,(\der_u-K_i)\,.
\vv.2>
\eeq
The unital subalgebra of $\Uglnt$ generated by $B_{ij}^\kk$, \,$i=1\lc N$,
\,$j\in\Z_{>0}\>$, is called the {\it Bethe algebra\/} and denoted by $\B^\kk$.

\sskip
By \cite{T}, \cite{MTV1}, \cite{CT}, the algebra $\B^\kk$ is commutative,
and $\B^\kk$ commutes with the subalgebra $U(\h)\subset \Uglnt$.
If all $K_1\lc K_N$ coincide, then $\B^\kk$ commutes with the subalgebra
$\Ugln\subset\Uglnt$.

\sskip
As a subalgebra of $\Uglnt$, the algebra $\B^\kk$ acts on any $\glnt$-module
$M$. Since $\B^\kk$ commutes with $U(\h)$, it preserves the weight subspaces
of $M$. If all $K_1\lc K_N$ coincide, then $\B^\kk$ preserves the subspace
$M^{\sing}$ of singular vectors.

If $L$ is a $\B^\kk$-module, then the image of $\B^\kk$ in $\End(L)$ is called
the {\it Bethe algebra of\/} $L$.

For our purpose it is convenient to consider another set of generators of
the Bethe algebra $\B^\kk$ defined as follows. Let $x$ be a new variable and
\vvn.4>
\beq
\label{PsiK}
\varPsi^\kk(u,x)\,=\,
\Bigl(\>x^N+\sum_{i=1}^N\,B_i^\kk(u)\,x^{N-i}\,\Bigr)
\prod_{i=1}^N\,\frac1{x-K_i}\;=\,
1+\sum_{i=1}^\infty\,\varPsi_i^\kk(u)\,x^{-i}\,.
\vv.1>
\eeq
The series $\varPsi_i^\kk(u)$, \,$i\in\Z_{>0}$, are linear combinations of
the series $B_i^\kk(u)$, \,$i=1\lc N$, and vice versa. Write
\vvn-.3>
\beq
\label{PsiiK}
\varPsi_i^\kk(u)\,=\,\sum_{j=1}^\infty\,\varPsi_{ij}^\kk\,u^{-j}\,.
\vv-.1>
\eeq
Then $\varPsi_{ij}^\kk$, \,$i,j\in\Z_{>0}$, is a new set of generators
of the Bethe algebra $\B^\kk$.

\section{Classical Gaudin Hamiltonians on $\otimes_{a=1}^nV(z_a)$}
\label{main}
Recall that $V$ is the vector representation of the Lie algebra $\gln$.
Consider the tensor product $\otimes_{a=1}^nV(z_a)$ of evaluation
$\glnt$-modules. The series $e_{ij}(u)$ acts on $\otimes_{a=1}^nV(z_a)$ as
$\sum_{a=1}^ne_{ij}^{(a)}(u-z_a)^{-1}$, where $e_{ij}^{(a)}$ is the image of
$1^{\otimes(a-1)}\otimes e_{ij}\otimes1^{\otimes(n-a)}\in
\bigl(U(\gln)\bigr)^{\otimes n}$.

\goodbreak
We denote by $B_{ij}\,,\varPsi_{ij}\in\End(V^{\otimes n})$ the images of the
elements $B_{ij}^\kk\,,\varPsi_{ij}^\kk\in\Uglnt$\,. \,Set
\begin{alignat}2
B_i(u)\,&{}=\,\sum_{j=0}^\infty\,B_{ij}\>u^{-j}\,,\qquad
&&\D\,=\,\der_u^N+\sum_{i=1}^N\,B_i(u)\,\der_u^{N-i}\,,
\notag
\\[4pt]
\label{Psi}
\varPsi_i(u)\,&{}=\,\sum_{j=1}^\infty\,\varPsi_{ij}\,u^{-j}\,,
&&\varPsi(u,x)\,=\,1+\sum_{i=1}^\infty\,\varPsi_i(u)\,x^{-i}\,.
\end{alignat}

All of the series $B_i(u)$, $\varPsi_i(u)$ sum up to rational functions of $u$
with values in $\End(V^{\otimes n})$.
Set in addition
\vvn-.1>
\be
\varPsi_\sdag(x)=-\sum_{i=1}^\infty\varPsi_{i1}x^{-i}\,.
\vv-.4>
\ee

\begin{lem}
We have
\beq
\label{Psi12}
\varPsi_1(u)\,=\,-\,\sum_{a=1}^n\,\frac1{u-z_a}\;,\qquad
\varPsi_2(u)\,=\,\sum_{a=1}^n\,\frac1{u-z_a}\,\Bigl(
-\>H_a\>+\>\sum_{b\ne a} \frac1{z_a-z_b}\,\Bigr)\,,
\vv-.2>
\eeq
where
\vvn-.4>
\beq
\label{H}
H_a\,=\,\sum_{i=1}^N\kk_i\>e_{ii}^{(a)}\,+\>
\sum_{i,j=1}^N\,\sum_{b\ne a} \,\frac{e_{ij}^{(a)}e_{ji}^{(b)}}{z_a-z_b}
\vv.3>
\eeq
are the classical Gaudin Hamiltonians~\Ref{Ha}, and
\vvn.1>
\be
\varPsi_\sdag(x)\,=\,\sum_{i=1}^N\,\sum_{a=1}^n\,\frac{e_{ii}^{(a)}}{x-K_i}\ .
\ee
\end{lem}
\begin{proof}
The claim is straightforward. See also formula~(8.5) and Appendix~B in
\cite{MTV1}.
\end{proof}

To formulate our main result we introduce a diagonal matrix
\vvn.2>
\beq
\label{Z}
Z\,=\,\diag(z_1\lc z_n)
\eeq
and a matrix
\vvn-.6>
\beq
\label{Q}
Q = \left(\,
\begin{matrix}
\yh_1 & \dfrac{1}{z_2-z_1} & \dfrac{1}{z_3-z_1} &\dots & \dfrac{1}{z_n-z_1}
\\[14pt]
\dfrac{1}{z_1-z_2} & \yh_2 & \dfrac{1}{z_3-z_2} &{} \dots & \dfrac{1}{z_n-z_2}
\\[9pt]
{}\dots & {}\dots & {} \dots & \dots & \dots
\\[6pt]
\dfrac{1}{z_1-z_n} & \dfrac{1}{z_2-z_n}& \dfrac{1}{z_3-z_n}
&{} \dots & \yh_n
\end{matrix}\,\right)
\vv.4>
\eeq
depending on new variables $\yh_1\lc\yh_n$. Set
\vvn.3>
\beq
\label{psi}
\psi(u,x,z_1\lc z_n,\yh_1\lc\yh_n)\,=\,
\det\bigl(1-(u-Z)^{-1}\>(x-Q)^{-1}\bigr)\,,
\vv-.2>
\eeq
\be
\phi(x,z_1\lc z_n,\yh_1\lc\yh_n)\,=\,\det(x-Q)\,,\quad
\psi_\sdag(x,z_1\lc z_n,\yh_1\lc\yh_n)\,=\,\tr\bigl((x-Q)^{-1}\bigr)\,.
\vv.2>
\ee

\begin{thm}
\label{generate}
The Bethe algebra of $\otimes_{a=1}^nV(z_a)$ is generated by the classical
Gaudin Hamiltonians $H_1\lc H_n$. More precisely,
\vvn.3>
\be
\varPsi(u,x)\,=\,\psi(u,x,z_1\lc z_n,H_1\lc H_n)\,.
\vvgood
\vv-.2>
\ee
In particular,
\vvn-.6>
\beq
\label{eii}
\psi_\sdag(x,z_1\lc z_n,H_1\lc H_n)\,=\,
\sum_{i=1}^N\,\sum_{a=1}^n\,\frac{e_{ii}^{(a)}}{x-K_i}\ .
\eeq
\end{thm}

\vsk.2>
\begin{rem}
Since \;$\tr\bigl((x-Q)^{-1}\bigr)\>=\>\der_x\log\bigl(\det(x-Q)\bigr)$,
\,formula~\Ref{eii} can be written as
\be
\phi(x,z_1\lc z_n,H_1\lc H_n)\,=\,
\prod_{i=1}^N\,(x-K_i)^{\sum_{a=1}^ne_{ii}^{(a)}}\,.
\vv.2>
\ee
\end{rem}

\begin{rem}
The matrix \,$[Q,Z]+1$ \,has rank one.
For every distinct $z_1\lc z_n$ and every $h_1\lc h_n$,
the pair \;$(Q\>,Z)$ \,defines a point of the $n$-th Calogero-Moser space,
hence, a point of the adelic Grassmannian. The function
$e^{ux}\>\psi(u,x,z_1\lc z_n,\yh_1\lc\yh_n)$ is the stationary Baker-Akhiezer
function of that point, see Section~3 in~\cite{Wi}.
Theorem \ref{generate} says that the coefficients
$\psi_{ij}(z_1\lc z_n,H_1\lc H_n)$
of the stationary Baker-Akhiezer function,
\vvn-.2>
\be
e^{ux}\psi(u,x,z_1\lc z_n,H_1\lc H_n)\,=\,e^{ux}\>\Bigl(\>1\>+
\sum_{i,j=1}^\infty\,\psi_{ij}(z_1\lc z_n,H_1\lc H_n)\,u^{-j}\>x^{-i}\,\Bigr)
\vv-.2>
\ee
generate the Bethe algebra of $\otimes_{a=1}^nV(z_a)$.
More remarks on this subject see in Section~\ref{CM}.
\end{rem}

\begin{cor}
\label{cor spec}
For distinct real $\kk_1\lc\kk_N$, and distinct real $z_1\lc z_n$, the joint
spectrum of the classical Gaudin Hamiltonians $H_1\lc H_n$ acting on
$\otimes_{a=1}^nV(z_a)$ is simple. That is, the classical Gaudin Hamiltonians
have a joint eigenbasis, and for any two vectors of the eigenbasis at least
one of the classical Gaudin Hamiltonians has different eigenvalues for those
vectors.
\end{cor}
\begin{proof}
By \cite{MTV5}, for distinct real $\kk_1\lc\kk_N$, and distinct real
$z_1\lc z_n$, the Bethe algebra of $\otimes_{a=1}^nV(z_a)$ has simple
spectrum. Therefore, the classical Gaudin Hamiltonians have simple spectrum
by Theorem~\ref{generate}.
\end{proof}

\begin{cor}
\label{cor spec0}
If $K_1\lc K_N$ coincide, and $z_1\lc z_n$ are distinct and real, then
the joint spectrum of the classical Gaudin Hamiltonians $H_1\lc H_n$
acting on $(\otimes_{a=1}^nV(z_a))^{\sing}$ is simple.
\end{cor}
\begin{proof}
By \cite{MTV3}, if $K_i=0$ for all $i=1\lc N$, and $z_1\lc z_n$ are real and
distinct, then the Bethe algebra of $(\otimes_{a=1}^nV(z_a))^{\sing}$ has
simple spectrum. Therefore, the classical Gaudin Hamiltonians acting
on $(\otimes_{a=1}^nV(z_a))^{\sing}$ have simple spectrum by
Theorem~\ref{generate}. The case of nonzero coinciding $K_1\lc K_N$ follows
from the case of zero $K_1\lc K_N$, since $\sum_{i=1}^N e_{ii}^{(a)}=1$ for
all $a=1\lc n$, see~\Ref{H}.
\end{proof}

\section{Proof of Theorem~\ref{generate}}
\subsection{Preliminary lemmas}
For functions $f_1(x)\lc f_m(x)$ of one variable, denote by
\vvn.2>
\be
\Wr[f_1\lc f_m]\,=\,\det\left( \begin{matrix}
f_1 & f_1'& \dots & f_1^{(m-1)}\;\\
f_2 & f_2'& \dots & f_2^{(m-1)}\\
\dots & \dots &\dots &\dots \\
f_m & f_m'& \dots & f_m^{(m-1)}
\end{matrix}\right)
\vv-.4>
\ee
the Wronskian of $f_1(x)\lc f_m(x)$.
\vvgood
\ Set \ $\dsty\Dl\,=\,\prod_{1\le a<b\le n}(z_b-z_a)$\,,\quad
$\dsty P(u)\,=\,\prod_{a=1}^n\,(u-z_a)$\,, \ and
\be
P_a(u)\,=\,\prod_{\fratop{b=1}{b\ne a}}^n\,\frac{u-z_b}{z_a-z_b}\,,\qquad
a=1\lc n\,.
\ee
Let $f_a(x) = (x+\mu_a)\>e^{z_ax}$, \,$a=1\lc n$, \,where \,$\mu_1\lc\mu_n$
are new variables. Set
\vvn-.1>
\be
W(u,x)\,=\,e^{-ux-\sum_{a=1}^n z_ax}\,
\Wr\bigl[f_1(x)\lc f_n(x), e^{ux}\bigr]\,=\,
W_0(x)\>\Bigl(u^n+\>\sum_{a=1}^n\,C_a(x)\>u^{n-a}\>\Bigr)\,.
\vv-.3>
\ee
Clearly, \ $W_0(x)\,=\,e^{-\sum_{a=1}^n z_ax}\,\Wr\bigl[f_1(x)\lc f_n(x)]$\,.

\begin{lem}
\label{WW}
Let \ $\dsty \yh_a\,=\,-\>\mu_a-\sum_{b\neq a}\,\frac1{z_a-z_b}$\;,\quad
$a=1\lc n$.\quad Then
\vvn.1>
\beq
\label{Wux}
W(u,x)\,=\,\Dl\cdot\det\bigl((u-Z)\>(x-Q)-1\bigr)\,,
\vv.2>
\eeq
where the matrices $Z$ and $Q$ are given by~\Ref{Z} and~\Ref{Q}.
In particular,
\vvn.3>
\beq
\label{W0}
W_0(x)\,=\,\Dl\cdot\det(x-Q)\,.
\vv.2>
\eeq
\end{lem}
\begin{proof}
First, we prove formula~\Ref{W0}. Let $S$ and $T$ be \>${n\times n}$ matrices
with entries \>$S_{ab}=z_b^{a-1}$ \>and \>$T_{ab}=(a-1)\>z_b^{a-2}$,
respectively. Clearly, \,$\det S=\Dl$. The entries of the matrix $S^{-1}$ are
determined by the equality \,$P_a(u)=\sum_{b=1}^n(S^{-1})_{ab}\,u^{b-1}$,
\>so that the entries of \,$S^{-1}T$ \,are \,$(S^{-1}T)_{ab}\>=\>P_a'(z_b)$.
\vsk.2>
Let \,$M=\diag(\mu_1\lc\mu_n)$. Since
\,$\der_x^k f_a(x)\>=\>\bigl((x+\mu_a)z_a^k+kz_a^{k-1}\bigr)\>e^{z_ax}$,
we have
\vvn.4>
\be
W_0(x)\,=\,\det\bigl(S\>(x+M)+T\bigr)\,=\,
\det S\cdot\det(x+M+S^{-1}T)\,=\,\Dl\cdot\det\bigl(x-Q)\,.
\ee

\vsk.3>
To prove formula~\Ref{Wux}, set $z_{n+1}=u$. Let \,$\Qh$ be
an \>${(n+1)\>{\times}\>(n+1)}$ \>matrix with entries $\Qh_{ab}=(z_b-z_a)^{-1}$
for $a\ne b$, \>and
\vvn-.6>
\be
\Qh_{aa}\,=\,-\>\mu_a-\sum_{\fratop{b=1}{b\neq a}}^{n+1}\,\frac1{z_a-z_b}\;,
\ee
where \>$\mu_{n+1}$ is a new variable.
\,Set \,$f_{n+1}(x)=(x+\mu_{n+1})\>e^{z_{n+1}x}$.
Similarly to~\Ref{W0}, we have
\vvn.3>
\be
e^{-\sum_{a=1}^{n+1}z_ax}\,\Wr\bigl[f_1(x)\lc f_{n+1}(x)\bigr]\,=\,
\Dl\cdot P(z_{n+1})\,\det(x-\Qh\>)\,.
\vv.3>
\ee
It is easy to see that
\vv.2>
\;$\Wr\bigl[f_1(x)\lc f_n(x), e^{ux}\bigr]\,=\,\lim_{\mu_{n+1}\to\infty}
\bigl(\>\mu_{n+1}^{-1}\Wr\bigl[f_1(x)\lc f_{n+1}(x)\bigr]\bigr)$
\ and
\ $\lim_{\mu_{n+1}\to\infty}\bigl(\>\mu_{n+1}^{-1}
\>\det(x-\Qh\>)\bigr)\>=\>\det\bigl(x-Q-(u-Z)^{-1}\bigr)$\,.
\;Then
\vvn.4>
\be
W(u,x)\,=\,\Dl\cdot P(u)\,\det\bigl(x-Q-(u-Z)^{-1}\bigr)\,=\,
\Dl\cdot\det\bigl((u-Z)\>(x-Q)-1\bigr)\,.
\vvn.2>
\ee
The lemma is proved.
\end{proof}

The complex vector space spanned by the functions $f_1\lc f_n$ is the kernel of
the monic differential operator
\vvn-.6>
\beq
\label{D}
D\,=\,\der_x^n\>+\sum_{a=1}^n\,C_a(x)\>\der_x^{n-a}\,.
\eeq
The function $\psi(u,x)$, defined by~\Ref{psi}, has the following expansion
as $u\to\infty$, \,$x\to\infty$:
\vvn.3>
\beq
\label{psiij}
\psi(u,x)\,=\,1\,+\>
\sum_{i=1}^\infty\,\sum_{j=1}^\infty\,\psi_{ij}\>u^{-j}\>x^{-i}\,.
\eeq
Here we suppressed the arguments \,$z_1\lc z_n$\>, \>$\yh_1\lc\yh_n$.
\,Set \,$\psi_i(u)\>=\sum_{j=1}^\infty\psi_{ij}\>u^{-j}$, \;$i\in\Z_{>0}$.

\begin{lem}
\label{expand}
We have
\vvn.3>
\beq
\label{psi12}
\psi_1(u)\,=\,-\,\sum_{a=1}^N\,\frac1{u-z_a}\;,\qquad
\psi_2(u)\,=\,\sum_{a=1}^N\,\frac1{u-z_a}\,
\Bigl(-\>\yh_a\>+\>\sum_{b\ne a} \frac1{z_a-z_b}\,\Bigr)\,.
\vv-.3>
\eeq
and
\vvn-.2>
\beq
\label{psix1}
\sum_{i=1}^\infty\,\psi_{i1}\,x^{-i}\,=\,-\tr\bigl((x-Q)^{-1}\bigr)\,.
\eeq
\end{lem}
\begin{proof}
The proof is straightforward from formulae~\Ref{Q},~\Ref{psi}.
\end{proof}

\subsection{Proof of Theorem~\ref{generate}}
Denote \;$\D_\reg=P(u)\,\D$.
By Theorem~3.1 in~\cite{MTV2}, we have
\vvn-.5>
\beq
\label{Dreg}
\D_{reg}=\sum_{i=0}^N\sum_{a=0}^nA_{ia}\>u^a\der^i\ ,
\qquad
A_{ia}\in \End(V^{\otimes n})\ ,
\vv-.2>
\eeq
and
\vvn-.4>
\be
\sum_{a=0}^n\,A_{Na}\>u^a\,=\,P(u)\,,\qquad
\sum_{i=0}^N\,A_{in}\>\der^i\,=\,R(\der_u)\,,\qquad
R(x)=\prod_{i=1}^N\,(x-K_i)\,.
\ee

\vsk.3>
Let $v \in \otimes_{a=1}^nV(z_a)$ be an eigenvector of the Bethe algebra,
$A_{ia}v= \al_{ia}v$, \,$\al_{ia}\in\C$, for all $(i,a)$.
Consider a scalar differential operator
\be
D_v\,=\,\sum_{i=0}^N\sum_{a=0}^n\,\al_{ia}\,x^i\der_x^a\ ,
\vv.2>
\ee
Notice that we changed \,$u\mapsto\der_x$, \;$\der_u\mapsto x$ compared
with~\Ref{Dreg}. By Theorem~3.1 in~\cite{MTV2} and Theorem~12.1.1
in~\cite{MTV4}, the kernel of $D_v$ is generated by the functions
$(x+\mu_a)\>e^{z_ax}$, \,$a=1\lc n$, \,with suitable $\mu_a\in\C$. Let
\be
\yh_a\,=\,-\>\mu_a-\sum_{b\neq a}\,\frac1{z_a-z_b}\;,\qquad a=1\lc n\,.
\vvgood
\ee

\begin{lem}
\label{lem on b}
We have \,$H_av=h_av$ \,for all \,$a=1\lc n$.
\end{lem}
\begin{proof}
We have \>$D_v=R(x)\,D$, where $D$ is given by~\Ref{D}. Then Lemma~\ref{WW}
and formulae~\Ref{PsiK}, \Ref{PsiiK}, \Ref{Psi} yield that the eigenvalues
of the operators $\varPsi_{ij}$ are the numbers $\psi_{ij}$ given
by~\Ref{psiij}: \,$\varPsi_{ij}v=\psi_{ij}v$. The
claim follows from comparing formulae~\Ref{Psi12} and~\Ref{psi12}.
\end{proof}

By Theorem~10.5.1 in~\cite{MTV4}, if $\kk_1\lc\kk_N$ and $z_1\lc z_n$
are generic, then the Bethe algebra of $\otimes_{a=1}^nV(z_a)$ has
an eigenbasis. Hence, by Lemmas~\ref{Wux} and~\ref{lem on b}, for such
$\kk_1\lc\kk_N$, $z_1\lc z_n$ we have
\be
\varPsi(u,x)\,=\,\psi(u,x,z_1\lc z_n,H_1\lc H_n)\,.
\vv.4>
\ee
Since both sides of this equality are meromorphic functions of
$\kk_1\lc\kk_N$ and $z_1\lc z_n$, the equality holds for
all $\kk_1\lc\kk_N$, $z_1\lc z_n$. The theorem is proved.
\qed

\section{Bethe algebra and functions on the Calogero-Moser space}
\label{CM}

\subsection{Calogero-Moser space $\Cc_n$}
Let \>$\Mc_n$ be the space of \,${n\times n}$ complex matrices.
The group $GL_n$ acts on \>$\Mc_n\oplus\Mc_n$ by conjugation,
\,$g:(X,Y)\mapsto(gXg^{-1},\>gYg^{-1})$\>.
\vvn.16>
Denote \,$\Fh_n=\C\>[\Mc_n\oplus\Mc_n]^{GL_n}$.

\smallskip
Let \,$\Cc_n\subset\Mc_n\oplus\Mc_n$ be the subset of pairs $(X,Y)$ with
the matrix $[X,Y]+1$ having rank one. The set \,$\Cc_n$ is $GL_n$-invariant.
The algebra $\Fc_n=\Fh_n|_{\Cc_n}$ is, by definition, the algebra of functions
on the $n$-th Calogero-Moser space, see~\cite{Wi}.

\smallskip
Consider a function
\vvn-.1>
\beq
\label{phi}
\pho(u,x,X,Y)\,=\,\det\bigl(1-(u-Y)^{-1}\>(x-X)^{-1}\bigr)\,,
\vv.3>
\eeq
depending on matrices $X,Y$ and variables $u,x$.
It has an expansion as $u\to\infty$, \,$x\to\infty$:
\vvn.2>
\beq
\label{phiij}
\pho(u,x,X,Y)\,=\,1\,+\>
\sum_{i=1}^\infty\,\sum_{j=1}^\infty\,\pho_{ij}(X,Y)\>u^{-j}\>x^{-i}\,
\vv-.3>
\eeq
with \,$\pho_{ij}\in\Fh_n$ \,for any $(i,j)$.

\smallskip

\begin{lem}[\cite{MTV6}]
\label{EG}
The algebra \,$\Fc_n$ is generated by the images of \,$\pho_{ij}$,
\,$i,j\in\Z_{>0}$.
\end{lem}

\subsection{Bethe algebra and functions on $\Cc_n$}
In this section we treat $\kk_1\lc\kk_N$ and $z_1\lc z_n$ as variables.
Set
\vvn.1>
\be
\Ec_{N,n}\>=\,\End(V^{\ox n})\ox\C\>[\kk_1\lc\kk_N,z_1\lc z_n]\;.
\vv.3>
\ee
We identify the algebras $\End(V^{\ox n})$ and $\C\>[\kk_1\lc\kk_N,z_1\lc z_n]$
with the respective subalgebras $\End(V^{\ox n})\ox 1$ and
$1\ox\C\>[\kk_1\lc\kk_N,z_1\lc z_n]$ \,of \,$\Ec_{N,n}$.

The operators $B_{ij}$ and $\Psi_{ij}$, defined in Section~\ref{main},
depend on $K_1\lc K_N$, $z_1\lc z_n$ polynomially, so we consider them
as elements of \,$\Ec_{N,n}$. Denote by \,$\Bc_{N,n}$ the unital
subalgebra of \,$\Ec_{N,n}$ generated by $B_{ij}$, \,$i=1\lc N$,
\,$j\in\Z_{>0}\>$.

\begin{lem}
\label{Bc}
The algebra \,$\Bc_{N,n}$ is generated by \,$\varPsi_{ij}$, \,$i=1\lc N$,
\,$j\in\Z_{>0}\>$, and symmetric polynomials in $\kk_1\lc\kk_N$.
\end{lem}
\begin{proof}
By formula~\Ref{Bi0}, we have
\be
x^N\>+\>\sum_{i=1}^N B_{i0}\,x^{N-i}\>=\,\prod_{i=1}^N\,(x-K_i)\,,
\vv.2>
\ee
so symmetric polynomials in $\kk_1\lc\kk_N$ belong to \,$\Bc_{N,n}$.
Formula~\Ref{PsiK} yields
\vvn.2>
\be
\Bigl(\>x^N+\sum_{i=1}^N\,\sum_{j=0}^\infty B_{ij}\,u^{-j}\>x^{N-i}\,\Bigr)
\prod_{i=1}^N\,\frac1{x-K_i}\;=\,
1+\sum_{i=1}^\infty\,\sum_{i=1}^\infty\,\varPsi_{ij}\,u^{-j}\>x^{-i}\,.
\vv.1>
\ee
Therefore, the elements \,$\varPsi_{ij}$ are linear combinations of
the elements \,$B_{ij}$ with coefficients being symmetric polynomials
in $\kk_1\lc\kk_N$, and vice versa. That proves the claim.
\end{proof}

Let \,$Z\>,Q$ \,be the matrices given by~\Ref{Z}, \Ref{Q}. For any $f\in\Fh_n$,
\>define a function $\bar f$ of the variables \,$z_1\lc,z_n$, $h_1\lc h_n$
\,by the formula
\vvn.3>
\be
\bar f(z_1\lc z_n,\yh_1\lc\yh_n)\,=\,f(Q,Z)\,.
\vv.2>
\ee

\begin{lem}
\label{fb}
The function \,$\bar f$ depends only on the image of \,$f$ in \,$\Fc_n$.
\end{lem}
\begin{proof}
The matrix \,$[Q,Z]+1$ \,has rank one, so the pair $(Q,Z)$ belongs to
\,$\Cc_n$,
\end{proof}

\begin{thm}
\label{second}
For any \,$f\in\Fh_n$, we have $\bar f(z_1\lc z_n,H_1\lc H_n)\in\Bc_{N,n}$.
In particular, $f(z_1\lc z_n,H_1\lc H_n)$ is a polynomial in
$z_1\lc z_n$.
\end{thm}
\begin{proof}
By Lemmas~\ref{fb} and~\ref{EG}, it suffices to prove the claim for the
functions \,$\pho_{ij}(X,Y)$. Since \,$\bar\pho_{ij}=\psi_{ij}$ by~\Ref{phi},
\Ref{phiij}, \Ref{psi}, \Ref{psiij}, and
\,$\psi_{ij}(z_1\lc z_n,H_1\lc H_n)=\varPsi_{ij}$ \,by Theorem~\ref{generate},
the statement follows from Lemma~\ref{Bc}.
\end{proof}

\begin{example}
Let \,$N=n=2$. \,Then \,$Z=\diag(z_1,z_2)$\>,
\ $Q\>=\>\left(\! \begin{array}{cccc} \yh_1 & (z_2-z_1)^{-1}\\[2pt]
(z_1-z_2)^{-1}\! & \!\yh_2 \end{array} \!\right)$\,,
\vvn.3>
\be
H_1\,=\,\kk_1\>e_{11}^{(1)}\,+\,
\kk_2\>e_{22}^{(1)}\,+\>\frac\Om{z_1-z_2}\,, \qquad
H_2\,=\,\kk_1\>e_{11}^{(2)}\,+\,
\kk_2\>e_{22}^{(2)}\,+\>\frac \Om{z_2-z_1}\,,
\ee
\be
\Om\,=\,e_{11}^{(1)}e_{11}^{(2)}+\>e_{12}^{(1)}e_{21}^{(2)}+\>
e_{21}^{(1)}e_{12}^{(2)}+\>e_{22}^{(1)}e_{22}^{(2)}\,.
\vv.5>
\ee
Let \,$f(X,Y)=\tr(X^2)$. \,Then
\,$\bar f(z_1,z_2,H_1,H_2)\>=\>H_1^2+\>H_2^2-\>2\>(z_1-z_2)^{-2}$
\,is a polynomial in $z_1,z_2$.
\end{example}

\begin{rem}
It is known that $\Fh_n$ is spanned by the functions
\;$\tr(X^{m_1}Y^{m_2}X^{m_3}Y^{m_4}\cdots{})$, \,where $m_1,m_2,\ldots{}$
are nonnegative integers, see~\cite{W}.
\end{rem}

Theorems~\ref{generate} and~\ref{second} show that the assignment
\,$\gm:f\mapsto\bar f(z_1\lc z_n,H_1\lc H_n)$
\,defines an algebra homomorphism \,$\Fh_n\to\Bc_{N,n}$ \,that sends
\,$\pho_{ij}$ to \,$\varPsi_{ij}$. By Lemma~\ref{fb}, this homomorphism
factors through \,$\Fc_n$\,. By Lemma~\ref{Bc}, the image of \,$\Fh_n$ tensored
with the algebra of symmetric polynomials in $\kk_1\lc\kk_N$ generate
\,$\Bc_{N,n}$.

\medskip
We show in~\cite{MTV6} that for \>$n=N$, the homomorphism \,$\gm$ \,induces
an isomorphism
\vvn.1>
of \,$\Fc_N$ with the quotient
of \,$\Bc_{N,N}$ by
the relations \,$\varPsi_{i1}=\>-\sum_{j=1}^N\,K_j^{i-1}$, \,$i\in\Z_{>0}$\,.
In other words, let
\vvn-.1>
\be
(V^{\ox N})_{\bf 1}=\,
\bigl\{\>v\in V^{\ox N}\ \big|\ \tsum_{a=1}^Ne_{ii}^{(a)}\>v\,=\,v\,,
\ i=1\lc N\>\bigr\}\,.
\vv.3>
\ee
Each element of \,$\Bc_{N,N}$ induces an element of
\,${\End\bigl((V^{\ox N})_{\bf 1}\bigr)\ox\C\>[\kk_1\lc\kk_N,z_1\lc z_N]}$\,.
Then \,$\Fc_N$ is isomorphic to the image of \,$\Bc_{N,N}$ in
\,$\End\bigl((V^{\ox N})_{\bf 1}\bigr)\ox\C\>[\kk_1\lc\kk_N,z_1\lc z_N]$\,.

\end{document}